\documentclass{amsart}

\usepackage{amssymb}
\usepackage{color}
\usepackage{amsxtra} 
\usepackage{mathtools} 

\newtheorem{theorem}{Theorem}%[section]
\newtheorem{lem}[theorem]{Lemma}
\newtheorem{prop}[theorem]{Proposition}
 
\newtheorem*{theorem*}{Theorem} 

\theoremstyle{definition}

\theoremstyle{remark}
\newtheorem{remark}[theorem]{Remark}

\numberwithin{equation}{section}

\begin{document}

\title[Local ill-posedness of the Euler equations in $B^1_{\infty,1}$] {Local ill-posedness of the Euler equations 
in $B^1_{\infty,1}$}

%    Information for first author
\author{Gerard Misio\l ek}
%    Address of record for the research reported here
\address{Department of Mathematics, University of Colorado, Boulder, CO 80309-0395, USA 
and 
Department of Mathematics, University of Notre Dame, IN 46556, USA} 
%    Current address
%\curraddr{Department of Mathematics, University of Colorado, Boulder, CO 80309-0395, USA} 
\email{gmisiole@nd.edu} 
%    \thanks will become a 1st page footnote.
%\thanks{The first author was supported in part by NSF Grant \#000000.}

%    Information for second author
\author{Tsuyoshi Yoneda}
\address{Department of Mathematics, Tokyo Institute of Technology, Meguro-ku, Tokyo 152-8551, Japan} 
\email{yoneda@math.titech.ac.jp}
\thanks{The second author was partially supported by JSPS KAKENHI Grant Number 25870004.} 

%    General info

\date{\today} 

%\dedicatory{This paper is dedicated to our advisors.}


\begin{abstract}
We show that the incompressible Euler equations on $\mathbb{R}^2$ are not locally well-posed 
in the sense of Hadamard in the Besov space $B^1_{\infty,1}$. 
Our approach relies on the technique of Lagrangian deformations of Bourgain and Li. 
We show that the assumption that the data-to-solution map is continuous in $B^1_{\infty,1}$ 
leads to a contradiction with a well-posedness result in $W^{1,p}$ of Kato and Ponce. 
\end{abstract}

\maketitle

%%%%%%%%%%%%%%%%%%%%%%%%%%%%%%%%
\section{Introduction} 
\label{sec:Intro} 

The study of the Cauchy problem for the Euler equations 
\begin{align} \label{eq:Euler-u} 
&u_t + u{\cdot}\nabla u + \nabla\pi = 0, 
\qquad 
t \geq 0, \, x \in \mathbb{R}^n 
\\ \label{eq:Euler-uu} 
&\mathrm{div}\, u = 0 
\\ \label{eq:Euler-u-ic} 
&u(0) = u_0 
\end{align} 
has a long history going back to the works of Gyunter \cite{Gu}, Lichtenstein \cite{Li} 
and Wolibner \cite{Wo} in the late 1920's and 1930's. 
Tremendous progress has been made since those pioneering papers and we refer to 
several excellent monographs and surveys for example Majda and Bertozzi \cite{MB}, 
Constantin \cite{Co} or Bahouri, Chemin and Danchin \cite{BCD} 
for detailed accounts. Nevertheless, the problems related to the phenomenon of turbulence 
and persistence of smooth solutions in 3D for all time remain open. 
Furthermore, despite extensive studies of local well-posedness for the Euler equations 
our understanding of this problem especially in the cases of important borderline spaces 
including $C^1$, $\mathrm{Lip}$, $B^1_{p,q}$, $W^{n/p +1,p}$ etc. has also remained incomplete.
However, this picture is changing fast. 

Recall that a Cauchy problem is \textit{locally well-posed} in a Banach space $X$ (in the sense of Hadamard) 
if for any initial data in $X$ there exist $T>0$ and a unique solution which persists in the space $C([0,T), X)$ 
and which depends continuously on the data. Otherwise, the problem is said to be ill-posed. 
 
A few years ago Bardos and Titi \cite{BT} used a shear flow of DiPerna and Majda \cite{DM} 
to construct solutions in 3D with an instantaneous loss of regularity in H\"{o}lder $C^\alpha$ 
and Zygmund $B^1_{\infty, \infty}$ spaces. More precisely, they found $C^\alpha$ initial data 
for which the corresponding (weak) solution does not belong to $C^\beta$ for any $1>\beta > \alpha^2$ 
and any $t>0$. This technique has also been used to obtain similar results 
in the Triebel-Lizorkin $F^1_{\infty, 2}$ space by Bardos, Lemarie and Titi and 
in the logarithmic Lipschitz spaces $\mathrm{logLip}^\alpha$ by the authors \cite{MY}.
 
More recently, in a breakthrough paper Bourgain and Li \cite{BL} employed both Lagrangian and Eulerian 
techniques to obtain strong local ill-posedness results 
in borderline Sobolev spaces $W^{n/p+1,p}$ for any $1 \leq p < \infty$ and 
in Besov spaces $B^{n/p+1}_{p,q}$ with $1 \leq p < \infty$ and $1 < q \leq \infty$ and $n=2$ or $3$. 
%In particular, they settled the borderline Sobolev case $H^{n/2+1}$. 
% 
In \cite{MY1} the authors adapted the approach of \cite{BL} to settle (in the 2D case) 
a long standing open question of local ill-posedness in the classical $C^1$ space 
by showing that the assumption on continuity of the data-to-solution map in $C^1$ leads to 
a contradiction with a results of Kato and Ponce \cite{KP} for $W^{1,p}$. 
Almost simultaneously Elgindi and Masmoudi \cite{EM} and Bourgain and Li \cite{BL1} produced 
similar results using different methods. It now seems that a complete resolution of local ill-posedness 
questions for the Euler equations including various borderline spaces is fully within reach. 
In fact, the main results of \cite{BL1} show that the Euler equations are ill-posed in the $C^m$ spaces 
for any integer $m \geq 1$ which is surprising in light of the local well-posedness results 
for the Cauchy problem \eqref{eq:Euler-u}-\eqref{eq:Euler-u-ic} in $C^{1, \alpha}$ for any $0 < \alpha <1$. 

The goal of this paper is to settle the end-point case of the Besov space of smoothness order $1$ 
and infinite (primary) integrability index. 
Recall that according to a recent result of Pak and Park \cite{PP} the Cauchy problem 
\eqref{eq:Euler-u}-\eqref{eq:Euler-u-ic} admits a unique solution in $B^1_{\infty, 1}(\mathbb{R}^n)$. 
It is interesting to observe that in order to establish uniqueness they first show that 
the data-to-solution map $u_0 \to u$ is continuous (even Lipschitz) into $B^0_{\infty,1}$ (cf. \cite{PP}, Section 4). They do not prove that it is continuous into $B^1_{\infty, 1}$ (and consequently that the Euler equations 
are locally well-posed in the sense of Hadamard in $B^1_{\infty, 1}$). 
Our main result is 
\begin{theorem} \label{thm:0} 
The 2D incompressible Euler equations are locally ill-posed in the Besov space $B^1_{\infty, 1}$.  
\end{theorem} 
As in our previous paper \cite{MY1} 
we will work with the vorticity equations. 
Recall that in two dimensions the vorticity of a vector field $u$ is a 2-form 
$\omega = d u^\flat$ which is identified with the function 
$$ 
\omega = \mathrm{rot}\, u = - \frac{\partial u_1}{\partial x_2} + \frac{\partial u_2}{\partial x_1}. 
$$ 
In this case the Cauchy problem \eqref{eq:Euler-u}-\eqref{eq:Euler-u-ic} can be rewritten as 
\begin{align} \label{eq:euler-v} 
&\omega_t + u{\cdot}\nabla \omega = 0, 
\qquad 
t \geq 0, \; x \in \mathbb{R}^2 
\\  \label{eq:euler-vic} 
&\omega(0) = \omega_0 
\end{align} 
where the velocity is recovered from $\omega$ using the Biot-Savart law 
\begin{equation} \label{eq:Biot-Savart} 
u = K \ast \omega = \nabla^\perp \Delta^{-1} \omega
\end{equation} 
with kernel $K(x) {=} (2\pi)^{-1}(-x_2/|x|^2, x_1/|x|^2)$ and where 
$\nabla^\perp {=} (-\frac{\partial}{\partial x_2}, \frac{\partial}{\partial x_1})$ 
denotes the symplectic gradient of a function. 

Our strategy is similar to that adopted for the $C^1$ case in \cite{MY1}. 
Namely, following \cite{BL} we first choose an initial vorticity $\omega_0$ such that 
the Lagrangian flow of the corresponding velocity field retains a large gradient 
on a (possibly short) time interval. 
We then perturb $\omega_0$ to get a sequence of initial vorticities in $W^{1,p}$. 
Finally, we show that the assumption that the Euler equations are well-posed 
in $B^1_{\infty.1}(\mathbb{R}^2)$ 
(in particular, that the solutions depend continuously on the initial data) 
leads to a contradiction with the following result of Kato and Ponce 
\begin{theorem*}[Kato-Ponce \cite{KP}] 
Let $1{<}p{<}\infty$ and $s{>}1{+}\frac{2}{p}$. For any $\omega_0 \in W^{s-1, p}(\mathbb{R}^2)$ 
and any $T>0$ there exists a constant $K=K(T,\omega,s,p)>0$ such that 
$$
\sup_{0 \leq t \leq T}\| \omega(t)\|_{W^{s-1,p}} \leq K. 
$$ 
\end{theorem*} 

Theorem \ref{thm:0} will be a consequence of the following result 
\begin{theorem} \label{thm:1} 
Let $2 < p < \infty$. Assume that the vorticity equations 
\eqref{eq:euler-v}-\eqref{eq:euler-vic} are well-posed in $B^0_{\infty,1}(\mathbb{R}^2)$. 
There exist $T>0$ and a sequence $\omega_{0,n}$ in $C^\infty_c(\mathbb{R}^2)$ 
with the following properties 

1. there exists a constant $C>0$ such that 
$\| \omega_{0,n} \|_{W^{1.p}} \leq C$ for all $n \in \mathbb{Z}_+$ 

\noindent and 

2. for any $M \gg 1$ there is $0 < t_0 \leq T$ such that 
$\| \omega_{n}(t_0)\|_{W^{1,p}} \geq M^{1/3}$ 
for all sufficiently large $n$ and all $p$ sufficiently close to $2$. 
\end{theorem} 
In Section \ref{sec:Lag} we provide some technical tools 
and construct an initial vorticity whose Lagrangian flow has a large gradient. 
Since some of the constructions are analogous to those in \cite{MY1} some details are omitted. 
The proof of Theorem \ref{thm:1} is given in Section \ref{sec:proof}. 

\begin{remark} 
In this paper we %use vorticity with noncompact support and 
do not employ the "patching" argument of \cite{BL} which leaves open the question of 
strong ill-posedness in $B^1_{\infty, q}$ in the sense of Bourgain-Li. 
\end{remark} 
\begin{remark} 
Since we treat here non-decaying data, an analogous local ill-posedness result in 3D follows 
immediately from our 2D construction. The details will be elaborated elsewhere. 
\end{remark} 
% 

%%%%%%%%%%%%%%%%%%%%%%%%%%%%%%%%%%%
\section{Vorticity and the Lagrangian flow} 
\label{sec:Lag} 

We first recall some basic harmonic analysis. 
Let $\psi$ be a smooth radial bump function on $\mathbb{R}^2$ which is supported in the unit ball $B(0,1)$ 
and equal to 1 on the ball of radius 1/2. Set $\psi_{-1}=\psi$ and let 
\begin{equation}  \label{eq:PL} 
\psi_0(\xi) = \psi(2^{-1}\xi) - \psi(\xi) 
\quad 
\mathrm{and} 
\quad 
\psi_\ell(\xi) = \psi_0 ( 2^{-\ell}\xi )
\quad 
\forall{\ell \geq 0}. 
\end{equation} 
Each $\psi_\ell$ is supported in a shell $\{ 2^{\ell-1}\leq |\xi| \leq 2^{\ell+1} \}$ 
with $\psi_\ell(\xi)=1$ when $|\xi|=2^\ell$. For any $f \in \mathcal{S}'(\mathbb{R}^2)$ define 
the frequency restriction operators by 
\begin{equation*} 
\widehat{\Delta_\ell f}(\xi) = \psi_\ell(\xi) \hat{f}(\xi) 
\quad 
\forall{\ell \geq -1}
\end{equation*} 
to obtain the usual Littlewood-Paley decomposition 
\begin{equation*} 
f = \sum_{\ell \geq -1} \Delta_\ell f 
\quad 
\mathrm{where}
\quad 
\Delta_\ell f(x) 
= 
\sum_{\xi \in \mathbb{R}^2} \psi_\ell(\xi)\hat{f}(\xi) e^{i \langle \xi, x \rangle}, 
\quad 
x \in \mathbb{R}^2. 
\end{equation*}

For any $s \in \mathbb{R}$ and $1 \leq p, q \leq \infty$ the Besov space $B^s_{p,q}(\mathbb{R}^2)$ 
is defined as the set of all $f \in \mathcal{S}'(\mathbb{R}^2)$ such that the number 
\begin{equation} \label{Besov}
\| f \|_{B^s_{p,q}} 
= \left\{ 
\begin{matrix} \displaystyle
~~~\Bigg( \sum_{\ell \geq -1}2^{sq\ell}\|\Delta_\ell f\|_{L^p}^q \Bigg)^{1/q}& 
\mathrm{if} \quad 1\leq q <\infty 
\\  \displaystyle
\sup_{\ell \geq -1}{ 2^{s\ell}} \|\Delta_k f\|_{L^p}& 
\mathrm{if} \qquad\;\; q=\infty 
\end{matrix} 
\right.
\end{equation}
is finite. 
Among many special cases of interest are the Sobolev spaces 
$W^{s,p}= B^s_{p,p}$ 
and the H\"older-Zygmund class 
$C^s = B^s_{\infty,\infty}$ both defined for any $1 \leq p < \infty$ and 
any non-integer $s > 0$. 

Next, given a radial bump function $0 \leq \varphi \leq 1$ supported in $B(0,1)$ define 
\begin{align} \label{eq:bump} 
\varphi_0(x_1, x_2) 
= 
\sum_{\varepsilon_1, \varepsilon_2 = \pm 1} 
\varepsilon_1 \varepsilon_2 \varphi(x_1 {-} \varepsilon_1, x_2 {-} \varepsilon_2) 
\end{align} 
and for a fixed positive integer $N_0 \in \mathbb{Z}_+$ and any $M \gg 1$, set 
\begin{align} \label{eq:iv} 
\omega_0(x) = \omega_0^{M,N}(x) 
= 
M^{-2} N^{-\frac{1}{p}} \sum_{N_0 \leq k \leq N_0+N} \varphi_k(x), 
\qquad 
N = 1, 2, 3 \dots 
\end{align} 
where $2< p < \infty$ and where 
\begin{align*} 
\varphi_k(x) = 2^{(-1 + \frac{2}{p})k} \varphi_0 (2^k x). 
\end{align*} 
Observe that by construction $\varphi_0$ is an odd function in both $x_1, x_2$ and 
for any $k \geq 1$ its support is compact and contained in the set 
\begin{equation} \label{eq:suppp} 
\mathrm{supp}\,{ \varphi_k } 
\subset 
\bigcup_{\varepsilon_1, \varepsilon_2 = \pm 1} 
B\big( (\varepsilon_1 2^{-k}, \varepsilon_2 2^{-k}), 2^{-(k+2)} \big). 
\end{equation} 
Combined with the uniform (in time) $L^\infty$ control of the vorticity in $\mathbb{R}^2$ 
this ensures the existence of a unique solution of the Cauchy problem \eqref{eq:euler-v}-\eqref{eq:euler-vic} 
with the initial data \eqref{eq:iv}; e.g., by a result of Yudovich \cite{Yu}, see also Majda and Bertozzi \cite{MB}. 

The construction so far parallels that of our previous paper \cite{MY1} and therefore the proofs of 
Lemma \ref{lem:omega-0} and Proposition \ref{prop:Lag} below will be omitted. 
\begin{lem} \label{lem:omega-0} 
We have 
\begin{align} \label{eq:omega-0} 
\| \omega_0 \|_{W^{1,p}} \lesssim  M^{-2} 
\end{align} 
with the bound independent of $N>0$ and $2<p<\infty$. 
\end{lem} 
\begin{proof} 
See \cite{MY1}; Lemma 3. 
%Since the supports in \eqref{eq:suppp} are disjoint  we have 
%% 
%\begin{align*} 
%\| \omega_0 \|_{L^p}^p 
%=
%M^{-2p} N^{-1} \sum_{N_0 \leq k \leq N_0 +N} 2^{-kp} \int_{\mathbb{R}^2} \big| \varphi_0(x) \big|^p dx 
%\lesssim 
%M^{-2p} 
%\end{align*} 
%% 
%and similarly 
%% 
%\begin{align*} 
%\Big\| \frac{\partial \omega_0}{\partial x_1} \Big\|_{L^p}^p 
%= 
%M^{-2p} N^{-1} \sum_{N_0 \leq k \leq N_0 +N} 
%\int_{\mathbb{R}^2} 2^{2k} \Big| \frac{\partial\varphi_0}{\partial x_1} (2^kx) \Big|^p dx. 
%\simeq 
%M^{-2p} 
%\end{align*} 
%%  
%The estimate of the other partial is analogous. 
\end{proof} 
Since $p>n=2$ the results of Kato and Ponce \cite{KP} (cf. Lemma 3.1; Thm. III) 
imply that there exists a unique velocity field $u \in C^1([0, \infty), W^{2,p}(\mathbb{R}^2))$ 
solving the problem \eqref{eq:Euler-u}-\eqref{eq:Euler-uu} and whose vorticity function 
$\omega \in C([0,\infty), W^{1,p}(\mathbb{R}^2))$ satisfies the initial condition \eqref{eq:iv}. 

The associated Lagrangian flow $\eta(t)$ of $u = \nabla^\perp\Delta^{-1}\omega$ 
is a solution of the initial value problem 
\begin{align} \label{eq:flow} 
&\frac{d}{dt}\eta(t,x) = u(t, \eta(t,x)) \; \big( {=} F_u(\eta(t,x)) \big) 
\\  \label{eq:flow-ic} 
&\eta(0,x) = x 
\end{align} 
and defines a curve in the group of volume-preserving diffeomorphisms such that  
$\omega\circ\eta \in C([0, \infty), W^{1,p}(\mathbb{R}^2))$, 
see e.g., \cite{KP} or \cite{BB}. 
It can be readily checked that the odd symmetry of $\omega_0$ is preserved by $\eta$ 
and thus (by conservation of vorticity in 2D) is retained by $\omega$ for all time. 
In this case the Biot-Savart law \eqref{eq:Biot-Savart} implies that the velocity field $v$ 
must be symmetric in the variables $x_1$ and $x_2$ so that both coordinate axes 
are invariant under $\eta$ with the origin $x_1=x_2=0$ a hyperbolic stagnation point. 

Observe that if 
$\xi: \mathbb{R}^2 \to \mathbb{R}^2$ is a volume-preserving diffeomorphism then 
the Jacobi matrix of its inverse $\xi^{-1}$ can be computed from 
\begin{align*} 
D\xi^{-1} 
= 
(D\xi)^{-1}\circ\xi^{-1} 
= 
\left( 
\begin{matrix} 
\partial_2\xi_2 {\circ} \xi^{-1} & -\partial_2 \xi_1{\circ} \xi^{-1} 
\\ 
-\partial_1 \xi_2{\circ}\xi^{-1} & \partial_1 \xi_1{\circ} \xi^{-1} 
\end{matrix} 
\right) 
\end{align*} 
so that for any smooth function $f: \mathbb{R}^2 \to \mathbb{R}$ we have 
\begin{equation} \label{eq:SGr} 
\nabla( f \circ \xi^{-1}) 
=  
\big( 
{-}\nabla{f} \circ \xi^{-1} \cdot \nabla^\perp{\xi_2} \circ \xi^{-1}, 
\nabla{f} \circ \xi^{-1} \cdot \nabla^\perp{\xi_1} \circ \xi^{-1} 
\big) 
\end{equation} 
where $\nabla^\perp$ is the symplectic gradient as in \eqref{eq:Biot-Savart}. 
\begin{prop} \label{prop:Lag} 
Let $\eta(t)$ be the flow of the velocity field $u=\nabla^\perp\Delta^{-1}\omega$ with initial vorticity 
given by \eqref{eq:iv}. Given $M \gg 1$ we have 
$$ 
\sup_{0 \leq t \leq M^{-3}} \| D\eta(t) \|_\infty > M 
$$ 
for any sufficiently large integer $N>0$ in \eqref{eq:iv} and any $2<p<\infty$ 
sufficiently close to $2$. 
\end{prop} 
\begin{proof} 
See \cite{MY1}; Section 3. 
\end{proof} 
In what follows it can be assumed without loss of generality that $2<p\leq3$. 
In this case all estimates on the flow $\eta$ or its derivative $D\eta$ can be made independent 
of the Lebesgue exponent $2<p < \infty$.

We will also need a comparison result for solutions of the Lagrangian flow equations, 
namely 
\begin{lem} \label{lem:comp} 
Let $u$ and $v$ be smooth divergence-free vector fields on $\mathbb{R}^2$ and let 
$\eta$ and $\xi$ be the solutions of \eqref{eq:flow}-\eqref{eq:flow-ic} with the right-hand sides 
given by $F_u$ and $F_{u+v}$ respectively. 
Then 
$$ 
\sup_{0 \leq t \leq 1}{ \big( \| \xi(t) - \eta(t) \|_\infty + \| D\xi(t) - D\eta(t) \|_\infty \big) } 
\leq 
C\sup_{0 \leq t \leq 1} ( \| v(t) \|_\infty + \| Dv(t)\|_\infty )  
$$ 
for some $C>0$ depending only on $u$ and its derivatives. 
\end{lem} 
\begin{proof} 
See e.g. \cite{BL}; Lemma 4.1. 
\end{proof}

\section{Proof of Theorem \ref{thm:1}} 
\label{sec:proof} 

Let $M \gg 1$ be an arbitrary large number and take $T = 1 \leq M^{-3}$. 
Recall that by assumption $2 < p < \infty$ and set $s=2$. 

Given the initial vorticity $\omega_0$ defined in \eqref{eq:iv} 
let $\omega(t)$ be the corresponding solution of the vorticity equations 
\eqref{eq:euler-v}-\eqref{eq:euler-vic} and let $\eta(t)$ be the associated Lagrangian flow of $u = \nabla^\perp\Delta^{-1}\omega$. 

If there exists $0 < t_0 \leq M^{-3}$ such that 
$\| \omega(t_0)\|_{W^{1,p}} > M^{1/3}$ 
then there is nothing to prove. 
We will therefore assume that 
\begin{equation} \label{eq:assump} 
\|\omega(t_0) \|_{W^{1,p}} \leq M^{1/3}, 
\qquad 
0 \leq t_0 \leq M^{-3}. 
\end{equation} 
Using Proposition \ref{prop:Lag} we can then find a point $x^\ast = (x^\ast_1, x^\ast_2)$ such that 
at least one of the entries $\partial\eta^i/\partial x_j$ of the Jacobi matrix 
(for example, the $i{=}j{=}2$ entry) satisfies $| \partial_2\eta_2 (t_0,x^\ast)| > M$. 
Therefore, by continuity, there is a $\delta >0$ such that 
\begin{align} \label{eq:M} 
\left| \frac{\partial \eta_2}{\partial x_2} (t_0,x) \right| > M 
\qquad 
\text{for all} 
\quad 
|x-x^\ast| < \delta. 
\end{align} 
Consider a smooth bump function $\hat{\chi} \in C^\infty_c(\mathbb{R}^2)$ in Fourier variables 
with support in the unit ball $B(0,1)$ and such that $0 \leq \hat{\chi} \leq 1$ 
and 
$\int_{\mathbb{R}^2} \hat\chi(\xi) \, d\xi = 1$. 
Let $\xi_0 = (2,0)$ and define 
\begin{equation} \label{eq:bp} 
\hat{\rho}(\xi) = \hat\chi(\xi - \xi_0) + \hat\chi(\xi + \xi_0), 
\qquad 
\xi \in \mathbb{R}^2 
\end{equation} 
Observe that the support of $\hat\rho$ is contained in $B(-\xi_0,1) \cup B(\xi_0,1)$ and 
\begin{equation} \label{eq:ro2} 
\rho(0) = \int_{\mathbb{R}^2} \hat{\rho}(\xi) \, d\xi = 2. 
\end{equation} 
For any $k \in \mathbb{Z}_+$ and $\lambda >0$ define 
\begin{equation} \label{eq:beta-pert} 
\beta_{k,\lambda}(x) 
= 
\frac{\lambda^{-1 + \frac{2}{p}}}{\sqrt{k}} 
\sum_{\varepsilon_1, \varepsilon_2 = \pm 1} 
\varepsilon_1 \varepsilon_2  \rho(\lambda(x-x^\ast_\epsilon)) \sin{kx_1} 
\end{equation} 
where 
$x^\ast_\epsilon = (\varepsilon_1 x^\ast_1, \varepsilon_2 x^\ast_2)$. 

To proceed we will need two technical lemmas. 
\begin{lem} \label{lem:rem} 
For any $k \in \mathbb{Z}_+$ and $\lambda >0$ we have 
\begin{enumerate} 
\item[1.] 
$
\| \partial_j \Delta^{-1} \beta_{k,\lambda} \|_\infty 
\lesssim 
k^{-1/2} \lambda^{-2 + \frac{2}{p}} \| \hat\rho\|_{L^1} 
$ 
\vskip 0.15cm 
\item[2.] 
$
\| \partial_i\partial_j \Delta^{-1} \beta_{k,\lambda} \|_\infty 
\lesssim 
k^{-1/2} \lambda^{-1 + \frac{2}{p}} \| \hat\rho\|_{L^1} 
$ 
\vskip 0.15cm 
\item[3.] 
$
\| \beta_{k,\lambda} \|_{W^{1,p}} 
\lesssim 
\big( k^{-1/2} + k^{1/2}\lambda^{-1} + k^{-1/2}\lambda^{-1} \big) \|\hat\rho\|_{L^{p'}} 
$ 
\end{enumerate} 
where $i, j = 1, 2$ and $1 < p' < 2$ is the conjugate exponent of $2 < p < \infty$. 
\end{lem} 
\begin{proof} 
Let $\xi_{\pm} = (\xi_1 \pm k/2\pi, \xi_2)$ and first compute the Fourier transform 
\begin{align} \nonumber 
\hat{\beta}_{k,\lambda}(\xi) 
= 
\frac{\lambda^{-1+\frac{2}{p}}}{\sqrt{k}} 
\sum_{\varepsilon_1, \varepsilon_2} 
\frac{\varepsilon_1 \varepsilon_2}{2i} 
\bigg( 
e^{-2\pi i \langle \xi_{-}, x^*_{\varepsilon} \rangle} 
&\frac{1}{\lambda^2} \hat{\rho}\Big( \frac{\xi_{-}}{\lambda} \Big) 
- 
\\  \label{eq:FTb} 
&- 
e^{-2\pi i \langle \xi_{+}, x^*_{\varepsilon} \rangle} 
\frac{1}{\lambda^2} \hat{\rho}\Big( \frac{\xi_{+}}{\lambda} \Big) 
\bigg). 
\end{align} 
Next, using the change of variable formula we estimate 
\begin{align*} 
\big| \partial_j \Delta^{-1} \beta_{k,\lambda}(x) \big| 
&\simeq 
\Big| \mathcal{F}^{-1}\mathcal{F} \big( \partial_j \Delta^{-1} \beta_{k,\lambda} \big) (x) \Big|
\lesssim 
\int_{\mathbb{R}^2} |\xi|^{-1} \big| \hat{\beta}_{k,\lambda}(\xi) \big| \, d\xi 
\\ 
& \lesssim 
k^{-1/2} \lambda^{-1+\frac{2}{p}} 
\int_{\mathbb{R}^2}  |\xi|^{-1} \frac{1}{\lambda^2} \Big( 
\big| \hat{\rho} (\lambda^{-1} \xi_{-} ) \big| 
+ 
\big| \hat{\rho} (\lambda^{-1} \xi_{+}) \big| 
\Big) d\xi 
\\ 
&\simeq 
k^{-1/2} \lambda^{-2+ \frac{2}{p}} \sum_{j=1,2} 
\int_{\mathbb{R}^2} |\xi|^{-1} \big| \hat{\rho}( \xi_1 + \tfrac{(-1)^j}{2\pi} \lambda^{-1}k, \xi_2) \big| \, d\xi. 
\end{align*} 
Since by construction for any sufficiently large $k \gg 10$ we have 
$$
\mathrm{supp}\, \hat\rho\big( \cdot \pm\tfrac{1}{2\pi} \lambda^{-1}k, \cdot \big) 
\cap 
B(0,1) = \emptyset 
$$ 
we can further estimate the above expression by 
\begin{align*} 
&\simeq 
k^{-1/2} \lambda^{-2+ \frac{2}{p}} \sum_{j=1,2} 
\int_{|\xi| \geq 1} \big| \hat{\rho}( \xi_1 + \tfrac{(-1)^j}{2\pi} \lambda^{-1}k, \xi_2) \big| \, d\xi 
\simeq 
k^{-1/2} \lambda^{-2+ \frac{2}{p}} \| \hat{\rho}\|_{L^1}. 
\end{align*} 

For the second assertion we similarly obtain 
\begin{align*} 
\big| \partial_j \Delta^{-1} \beta_{k,\lambda}(x) \big| 
%\simeq 
%\Big| \mathcal{F}^{-1}\mathcal{F} \big( \partial_j \Delta^{-1} \beta_{k,\lambda} \big) (x) \Big|
\lesssim 
\| \hat{\beta}_{k,\lambda}\|_{L^1} 
\lesssim 
k^{-1/2} \lambda^{-1 + \frac{2}{p}} \| \hat\rho \|_{L^1}. 
\end{align*} 

Finally, using the triangle inequality and the change of variables formula we get 
\begin{align*} 
\Big\| \frac{\partial\beta_{k,\lambda}}{\partial x_1} \Big\|_{L^p} 
&\lesssim 
\frac{1}{\sqrt{k}} \Big\| 
\lambda^{2/p} \sum_{\varepsilon_1,\varepsilon_2} 
\varepsilon_1 \varepsilon_2 \frac{\partial \rho}{\partial x_1}\big( \lambda(\cdot - x^*_\varepsilon) \big) 
\Big\|_{L^p} 
+ 
\frac{\sqrt{k}}{\lambda} \Big\| 
\lambda^{2/p} \sum_{\varepsilon_1,\varepsilon_2} 
\varepsilon_1 \varepsilon_2 \rho\big(\lambda(\cdot-x^*_\varepsilon)\big) 
\Big\|_{L^p} 
\\ 
&\simeq 
\frac{1}{\sqrt{k}} \left( \int_{\mathbb{R}^2} 
\Big| \sum_{\varepsilon_1, \varepsilon_2} 
\varepsilon_1\varepsilon_2 \frac{\partial\rho}{\partial x_1}(x) \Big|^p 
dx \right)^{1/p} 
+ 
\frac{\sqrt{k}}{\lambda} \left( \int_{\mathbb{R}^2} 
\Big| \sum_{\varepsilon_1, \varepsilon_2} 
\varepsilon_1\varepsilon_2 \rho(x) \Big|^p 
dx \right)^{1/p} 
\end{align*} 
and using Hausdorff-Young with $1/p'+ 1/{p'} =1$ we find 
\begin{align*} 
\lesssim 
k^{-1/2} \Big\| \frac{\partial \rho}{\partial x_1}\Big\|_{L^p} 
+ 
k^{1/2}\lambda^{-1} \| \rho\|_{L^p}  
&\lesssim 
k^{-1/2} \Big\| \widehat{\frac{\partial\rho}{\partial x_1}} \Big\|_{L^{p'}} 
+ 
k^{1/2} \lambda^{-1} \| \hat\rho \|_{L^{p'}} 
\\ 
&\lesssim 
\big( k^{-1/2} + k^{1/2}\lambda^{-1} \big) \| \hat\rho \|_{L^{p'}}. 
\end{align*} 

Similarly, we also obtain 
\begin{align*} 
\Big\| \frac{\partial\beta_{k,\lambda}}{\partial x_2} \Big\|_{L^p} 
\lesssim 
k^{-1/2} \Big\| \frac{\partial \rho}{\partial x_2} \Big\|_{L^p} 
\lesssim 
k^{-1/2} \| \hat\rho \|_{L^{p'}} 
\end{align*} 
and 
\begin{align*} 
\| \beta_{k,\lambda}\|_{L^p} 
\lesssim 
\frac{\lambda^{-1}}{\sqrt{k}} \left( \int_{\mathbb{R}^2} 
\Big| \lambda^{2/p} \sum_{\varepsilon_1, \varepsilon_2} 
\varepsilon_1\varepsilon_2 \rho\big( \lambda(x-x^\ast_\varepsilon) \big) \Big|^p dx \right)^{1/p} 
\lesssim 
k^{-1/2} \lambda^{-1} \| \hat\rho\|_{L^{p'}} 
\end{align*} 
which combined yield the lemma. 
\end{proof} 

Observe that choosing 
\begin{equation} \label{eq:kln} 
k= \lambda^2 
\quad \text{and} \quad 
\lambda = 3n, 
\quad 
n \gg 100 
\end{equation} 
and letting 
$\beta_n = \beta_{k,\lambda}$ 
in Lemma \ref{lem:rem} we immediately obtain 
\begin{align} \label{eq:n1} 
\| \nabla^\perp \Delta^{-1} \beta_n \|_\infty  \rightarrow 0 
\quad 
\text{and} 
\quad 
\| D \nabla^\perp \Delta^{-1} \beta_n \|_\infty \rightarrow 0 
\quad 
\text{as} 
\; 
n \to \infty 
\end{align} 
and 
\begin{align} \label{eq:n2} 
\| \beta_n\|_{W^{1,p}} \simeq \| \beta_n \|_{L^p} + \| \nabla\beta_n \|_{L^p} 
\lesssim \| \hat\rho \|_{L^{p'}} < \infty
\quad 
\text{for any} 
\; 
n \in \mathbb{Z}_+. 
\end{align} 

The second lemma we need is 
\begin{lem} \label{lem:remrem} 
Let $t_0 >0$ be as in \eqref{eq:assump} and let $k, \lambda$ and $n$ be as in \eqref{eq:kln}. 
Then 
\begin{enumerate} 
\item[1.] 
$
\| \partial_2 \beta_{k,\lambda} \partial_1 \eta_2(t_0) \|_{L^p} 
\lesssim 
n^{-1} C_{0,T} \| \hat\rho \|_{L^{p'}} 
\xrightarrow[n \to \infty]{} 0 
$ 
\vskip .08cm
\item[2.] 
$
\| \partial_1\beta_{k,\lambda} \partial_2\eta_2(t_0) \|_{L^p} 
\gtrsim 
M \big( \epsilon\pi + \mathcal{O}(n^{-1/2}) \big) 
- 
n^{-1} C_{0,T} \| \hat\rho \|_{L^{p'}} 
\xrightarrow[n \to \infty]{} M 
$ 
\end{enumerate} 
where $C_{0,T} = \exp^{ T\sup_{0 \leq t \leq T} \| D\nabla^\perp\Delta^{-1}\omega(t)\|_\infty } < \infty$. 
\end{lem} 
\begin{proof} 
From \eqref{eq:beta-pert} we have 
\begin{align} \nonumber 
\bigg( \int_{\mathbb{R}^2} \Big| 
&\frac{\partial \beta_{k,\lambda}}{\partial x_2}(x) \frac{\partial \eta_2}{\partial x_1}(t_0, x) 
\Big|^p dx \bigg)^{1/p} 
\\  \label{eq:hh} 
&= 
\left( \int_{\mathbb{R}^2} \bigg| 
\frac{1}{\sqrt{k}} \lambda^{\frac{2}{p}} \sum_{\varepsilon_1, \varepsilon_2} 
\varepsilon_1 \varepsilon_2 
\frac{\partial \rho}{\partial x_2}( \lambda(x-x^*_\varepsilon)) \sin{kx_1} \frac{\partial\eta_2}{\partial x_1}(t_0,x) 
\bigg|^p dx \right)^{1/p} 
\\  \nonumber 
&\leq 
\frac{1}{\sqrt{k}} \sum_{\varepsilon_1, \varepsilon_2} \left( 
\int_{\mathbb{R}^2} 
\lambda^2 \Big| \frac{\partial\rho}{\partial x_2}(\lambda(x-x^*_\varepsilon)) \Big|^p dx 
\right)^{1/p} 
\sup_{x \in \mathbb{R}^2} \Big| \frac{\partial\eta_2}{\partial x_1}(t_0,x) \Big|. 
\end{align} 
Differentiating the flow equations \eqref{eq:flow}-\eqref{eq:flow-ic} in the spatial variable and 
taking the $L^\infty$ norms we obtain a differential inequality 
%$$ 
%\frac{d}{dt} \| D\eta(t) \|_\infty \leq \| Du(t) \|_\infty \| D\eta(t) \|_\infty 
%\quad\text{and}\quad 
%\| D\eta(0) \|_\infty = 1 
%$$ 
which with the help of Grownwall's lemma gives 
$$ 
\| D\eta(t) \|_\infty \leq e^{ \int_0^t \| Du(\tau)\|_\infty d\tau } \leq C_{0,T}. 
$$ 
Using this bound the right hand side of \eqref{eq:hh} can be estimated further by 
$$ 
\simeq 
k^{-1/2} \| \partial_1\eta_2(t_0)\|_{\infty} \| \rho\|_{L^p} 
\lesssim 
k^{-1/2} C_{0,T} \| \hat\rho\|_{L^{p'}}. 
$$ 
which gives the first assertion. For the second one we have 
\begin{align*} 
\bigg( \int_{\mathbb{R}^2} \Big| 
&\frac{\partial \beta_{k,\lambda}}{\partial x_1}(x) \frac{\partial \eta_2}{\partial x_2}(t_0, x) 
\Big|^p dx \bigg)^{1/p} 
= 
\\ 
&= 
\Bigg( \int_{ \mathbb{R}^2 } \bigg| 
\frac{1}{\sqrt{k}} \lambda^{\frac{2}{p}-1} \sum_{\varepsilon_1, \varepsilon_2} 
\varepsilon_1 \varepsilon_2 
\Big( 
k \rho(\lambda(x-x^*_\varepsilon)) \cos{kx_1} 
\, + 
\\ 
& \hskip 3cm + 
\lambda \frac{\partial\rho}{\partial x_1} (\lambda(x - x^*_\varepsilon)) \sin{kx_1} 
\Big) 
\frac{\partial \eta_2}{\partial x_2}(t_0,x) 
\bigg|^p dx \Bigg)^{1/p} 
\\ 
&\geq 
\Bigg( 
\int_{B(x^*,\delta)} 
\bigg| 
\sqrt{k} \lambda^{-1 + \frac{2}{p}} \cos{kx_1} \rho(\lambda(x-x^*)) \frac{\partial\eta_2}{\partial x_2}(t_0,x) 
\, + 
\\ 
&\hskip 3cm + 
\frac{1}{\sqrt{k}} \lambda^{\frac{2}{p}} \sin{kx_1} \frac{\partial\rho}{\partial x_1}(\lambda(x-x^*)) 
\frac{\partial\eta_2}{\partial x_2}(t_0,x) 
\bigg|^p 
dx 
\Bigg)^{1/p}. 
\end{align*} 
Using the triangle inequality, \eqref{eq:M} and the change of variables formula we estimate 
the above integral from below by 
\begin{align}  \nonumber 
\gtrsim 
M \sqrt{k} \lambda^{-1} 
\bigg( 
\int_{B(x^*,\delta)} &\lambda^2 
\big| \cos{kx_1} \rho(\lambda(x-x^*)) \big|^p 
dx \bigg)^{1/p} 
- 
\\  \label{eq:lll} 
&- 
\frac{1}{\sqrt{k}} \bigg( 
\int_{B(x^*,\delta)} \lambda^2 
\Big| \sin{kx_1} \frac{\partial\rho}{\partial x_1}(\lambda(x-x^*)) \frac{\partial\eta_2}{\partial x_2}(t_0,x)\Big|^p 
dx \bigg)^{1/p} 
\\  \nonumber 
&\hskip - 3.4cm \geq 
M \sqrt{k} \lambda^{-1} 
\bigg( 
\int_{B(0,\lambda\delta)} \big| \cos{(k\lambda^{-1}x_1 + kx^\ast_1)} \big|^p |\rho(x)|^p 
dx \bigg)^{1/p} 
- 
\\  \nonumber 
&- 
\frac{1}{\sqrt{k}} \| \partial_1\rho \|_{L^p} \|\partial_2\eta_2(t_0)\|_{L^\infty(B(x^\ast,\delta))}. 
\end{align} 
We now focus on the integral term on the right hand side of \eqref{eq:lll}. 
From \eqref{eq:ro2} we have $\rho(0)=2$ so that by continuity there exists 
an $\epsilon >0$ such that 
$$
| \rho(x)| \geq 1, 
\qquad 
\text{for any~} 
x \in B(0,\epsilon). 
$$ 
Therefore, if we set $\delta = \epsilon/\lambda$ then the integral term can be bounded below by 
\begin{align*} 
\left( \int_{B(0,\epsilon)} \big| \cos{(k\lambda^{-1}x_1 + kx^\ast_1)} \big|^p dx \right)^{1/p} 
&\geq 
\left( \int_{-\epsilon\pi/6}^{\epsilon\pi/6} \int_{-\epsilon\pi/6}^{\epsilon\pi/6} 
\cos^2 (\lambda x {+} \lambda^2 x^*_1) \, dx dy \right)^{1/2} 
\\ 
&\simeq 
\frac{\epsilon\pi}{3\sqrt{2}} + \mathcal{O}(\lambda^{-1/2}) 
\end{align*} 
by a straightforward calculation using the assumption on $ p > 2$ and the choices of parameters 
made in \eqref{eq:kln}. 
To complete the proof of the lemma it suffice to observe that \eqref{eq:lll} can now be further bounded by 
$$
\gtrsim 
M \big( \epsilon\pi + \mathcal{O}(n^{-1/2}) \big) - \frac{C_{0,T}}{n} \| \hat\rho \|_{L^{p'}} 
$$ 
which is the required estimate. 
\end{proof} 

Consider the following sequence of initial vorticities 
\begin{equation} \label{eq:omega-seq} 
\omega_{0,n}(x) = \omega_0(x) + \beta_n(x), 
\qquad 
n \in \mathbb{Z}_{+}. 
\end{equation} 
From equations \eqref{eq:n2} and \eqref{eq:omega-0} of Lemma \ref{lem:omega-0} it follows that 
$\omega_{0,n}$ belongs to $W^{1,p}$ for any $n \in \mathbb{Z}_+$. 
Let $\omega_n(t)$ be the corresponding solutions of the vorticity equations 
\eqref{eq:euler-v}-\eqref{eq:euler-vic}. 
For each $n \in \mathbb{Z}_+$ (sufficiently large if necessary) let $\eta_n(t)$ be the flow of 
volume-preserving diffeomorphisms of the velocity fields 
$u_n = \nabla^\perp\Delta^{-1}\omega_n$. 

\medskip 
(A). Let $1 \leq q < \infty$ and 
assume that the data-to-solution map for the Euler equations is continuous 
from bounded sets in $B^1_{\infty,1}(\mathbb{R}^2)$ to $C([0,1], B^1_{\infty,1}(\mathbb{R}^2))$. 
\begin{lem} \label{lem:fi} 
For any $1 \leq q < \infty$ we have 
\begin{align*} 
\| \nabla^\perp \Delta^{-1} \beta_n \|_{B^1_{\infty,q}} 
\simeq 
\| \nabla^\perp \Delta^{-1} \beta_n \|_\infty 
+ 
\| D \nabla^\perp \Delta^{-1} \beta_n \|_\infty 
\end{align*} 
for any sufficiently large $n \in \mathbb{Z}_{+}$. 
\end{lem} 
\begin{proof} 
It will be sufficient to show the equalities 
$\|D^\alpha\nabla^\perp\Delta^{-1}\beta_n\|_{B^0_{\infty,q}} 
\simeq 
\|D^\alpha\nabla^\perp\Delta^{-1}\beta_n\|_\infty$ 
for $|\alpha| = 0$ and $1$. 

As in the proof of Lemma \ref{lem:rem} recall that for any large enough integer $n \in \mathbb{Z}_{+}$ 
we have 
$\mathrm{supp}\, \hat{\beta}_n \cap B(0,1) = \emptyset$
so that only need to establish 
$\| \beta_n \|_\infty \simeq \| \beta_n \|_{B^0_{\infty,q}}$. 
Let $\psi_\ell$ be the family of Paley-Littlewood functions as in \eqref{eq:PL} 
supported in the shell $\{ 2^{\ell -1} \leq |\xi| \leq 2^{\ell +1} \}$. 
From \eqref{eq:FTb} and the definition of $\rho$ we find that 
$$ 
\mathrm{supp}\, \hat{\beta}_{n=2^j} \subset B((2^{2j},0), 2^j)
$$ 
since $\mathrm{supp}\,\hat{\rho} \subset B(0,3)$, see \eqref{eq:bp}. 
It follows that 
for any $j \in \mathbb{Z}_{+}$ we can find an $\ell_j \in \mathbb{Z}_{+}$ such that 
$\mathrm{supp}\, \hat{\beta}_{2^j} \subset \mathrm{supp}\, \psi_{\ell_j}$ 
which implies that 
$$
\| \beta_{2^j} \|_{B^0_{\infty,q}} 
= 
\bigg( \sum_{\ell \geq -1} \|  \Delta_\ell \beta_{2^j} \|_\infty^q \bigg)^{1/q} 
\simeq 
\| \hat{\psi}_{\ell_j} {\ast} \beta_{2^j} \|_\infty 
\simeq 
\| \beta_{2^j} \|_\infty. 
$$ 
The other equality can be shown analogously. 
\end{proof} 

Combining \eqref{eq:n1}, \eqref{eq:omega-seq} and Lemma \ref{lem:fi} it follows now from 
the assumption (A) on the continuity of the solution map that 
\begin{equation} \label{eq:A} 
\sup_{0 \leq t \leq 1}\| \nabla^\perp\Delta^{-1} ( \omega_n(t) - \omega(t) ) \|_{B^1_{\infty,1}} 
\longrightarrow 0 
\quad 
\text{as} 
\quad 
n \to \infty 
\end{equation} 
and consequently by an elementary embedding $B^1_{\infty,1} \subset C^1$ 
we also have 
\begin{align*} 
\sup_{0 \leq t \leq 1}\| \nabla^\perp\Delta^{-1} ( \omega_n(t) - \omega(t) ) \|_{C^1} 
\longrightarrow 0 
\quad 
\text{as} 
\quad 
n \to \infty. 
\end{align*} 
Applying the comparison Lemma \ref{lem:comp} we now find 
\begin{align} \label{eq:LL} 
\sup_{0\leq t \leq 1} \big( 
\| \eta_n(t) - \eta(t) \|_\infty + \| D\eta_n(t) - D\eta(t) \|_\infty 
\big) 
= 
\theta_n 
\longrightarrow 0 
\quad 
\text{as} 
\; 
n \to \infty 
\end{align} 
where $\eta(t)$ is the flow of the velocity field $u=\nabla^\perp\Delta^{-1}\omega$ 
with the initial vorticity $\omega_0$ given by \eqref{eq:iv} as in Proposition \ref{prop:Lag}. 

Using conservation of vorticity, formula \eqref{eq:SGr} and the invariance of the $L^p$ norms 
under volume-preserving Lagrangian flows $\eta_n(t)$ we have 
\begin{align} \nonumber 
\qquad 
\| \omega_n(t_0) \|_{W^{1,p}} 
&\geq 
\| \nabla( \omega_{0,n}\circ\eta_n^{-1}(t_0)) \|_{L^p} 
\simeq 
\\  \nonumber 
&\hskip -2.8cm 
\| d\omega_{0,n}{\circ}\eta_n^{-1}(t_0) (\nabla^\perp\eta_{n,2}(t_0) {\circ}\eta_n^{-1}(t_0)) \|_{L^p} 
{+} 
\| d\omega_{0,n}{\circ}\eta_n^{-1}(t_0) (\nabla^\perp\eta_{n,1}(t_0){\circ}\eta_n^{-1}(t_0)) \|_{L^p} 
\\  \label{eq:BB} 
&\simeq 
\| d\omega_{0,n} (\nabla^\perp\eta_{n,2}(t_0)) \|_{L^p} 
+ 
\| d\omega_{0,n} (\nabla^\perp\eta_{n,1}(t_0)) \|_{L^p} 
\\ \nonumber 
&\gtrsim 
\| d\omega_{0,n} ( \nabla^\perp\eta_{n,2}(t_0)) \|_{L^p}. 
\end{align} 
Since from the comparison estimate \eqref{eq:LL} we have 
$$ 
\| d\omega_{0,n}( \nabla^\perp\eta_2 - \nabla^\perp\eta_{n,2})(t_0)\|_{L^p} 
\lesssim 
\| D( \eta_2 - \eta_{n,2} )(t_0)\|_\infty \|\nabla\omega_{0,n}\|_{L^p} 
\leq 
\theta_n \|\nabla\omega_{0,n}\|_{L^p} 
$$ 
applying the triangle inequality and \eqref{eq:omega-seq} we can further bound 
the right side of the expression in \eqref{eq:BB} below by 
\begin{align} \label{eq:MT}  
&\| d\omega_{0,n} (\nabla^\perp\eta_2(t_0)) \|_{L^p} 
- 
\theta_n \| \nabla\omega_{0,n}\|_{L^p} 
\\  \nonumber 
&\hskip 1cm \gtrsim 
\| d\beta_n (\nabla^\perp\eta_2(t_0)) \|_{L^p} 
- 
\| d\omega_0(\nabla^\perp\eta_2(t_0))\|_{L^p} 
- 
\theta_n \| \nabla\omega_{0,n}\|_{L^p}. 
\end{align} 
Observe that by the assumption \eqref{eq:assump} we can bound the middle term 
on the right side of \eqref{eq:MT} as in \eqref{eq:BB} above by 
\begin{align} \nonumber 
\| d\omega_0( \nabla^\perp\eta_2(t_0)) \|_{L^p} 
&\leq 
\| \nabla\omega_0\circ\eta^{-1}(t_0) \cdot D\eta^{-1}(t_0) \|_{L^p} 
\\  \label{eq:WW} 
&\simeq 
\| \nabla( \omega_0\circ\eta^{-1} (t_0) )\|_{L^p} 
\leq 
\| \omega(t_0) \|_{W^{1,p}} 
\leq M^{1/3}. 
\end{align} 

It therefore remains to find a lower bound on the $\beta$-term in \eqref{eq:MT}. 
This however follows from the the two estimates in Lemma \ref{lem:remrem}. 
Namely, we have 
\begin{align}  
\| d\beta_n( \nabla^\perp\eta_2(t_0)) \|_{L^p} 
&= 
\big\| -\partial_1\beta_n \partial_2\eta_2(t_0) + \partial_2\beta_n \partial_1\eta_2(t_0) \big\|_{L^p} 
\nonumber 
\\ \nonumber 
&\geq 
\| \partial_1\beta_n \partial_2\eta_2(t_0)\|_{L^p} 
- 
\| \partial_2 \beta_n \partial_1\eta_2(t_0)\|_{L^p} 
\\   \label{eq:bBeta} 
&\gtrsim 
M \big( \epsilon\pi + \mathcal{O}(n^{-1/2}) \big) 
- 
n^{-1} C_{0,T} \| \hat\rho \|_{L^{p'}} 
- 
n^{-1} C_{0,T} \| \hat\rho \|_{L^{p'}} 
\\  \nonumber 
&\gtrsim 
M 
\end{align} 
provided that $n$ is sufficiently large. 
Theorem \ref{thm:1} follows now by combining \eqref{eq:BB} with \eqref{eq:MT}, \eqref{eq:WW} 
and \eqref{eq:bBeta}.

%%%%%%%%%%%%%%%%%%%%%%%%%%%%%%%%%%
\bibliographystyle{amsplain}

\end{document}